\documentclass{article}
\pdfoutput=1
\usepackage{graphicx}

\usepackage{subfigure}
\usepackage{amsmath}
\usepackage{mathrsfs} 
\usepackage{amsfonts}
\usepackage{amssymb}%
\usepackage{amsthm}
\usepackage{amscd}
\usepackage{verbatim}
 \usepackage{mathabx}
 \usepackage{dsfont}
 \usepackage[all]{xypic} 
\usepackage{endnotes}
\usepackage{enumerate}
\usepackage{mathtools}
\usepackage{comment}
\usepackage{algorithm}
\usepackage{algcompatible}
\usepackage[noend]{algpseudocode}
\usepackage{tikz}
\usepackage{color}

\newcommand{\mcM}{\mathcal{M}}

\newcommand{\mcC}{\mathcal{C}}

\theoremstyle{plain}
\newtheorem{theorem}{Theorem}[section]
\newtheorem{corollary}[theorem]{Corollary}
\newtheorem{proposition}[theorem]{Proposition}
\newtheorem{lemma}[theorem]{Lemma}

\newtheorem{observation}[theorem]{Observation}

\theoremstyle{definition}
\newtheorem{defn}[theorem]{Definition}
\theoremstyle{definition}

\theoremstyle{definition}
\newtheorem{example}[theorem]{Example}
\theoremstyle{definition}
\theoremstyle{definition}

\theoremstyle{definition}

\usepackage{listings}
\usepackage{color}

\definecolor{dkgreen}{rgb}{0,0.6,0}
\definecolor{gray}{rgb}{0.5,0.5,0.5}
\definecolor{mauve}{rgb}{0.58,0,0.82}

\usepackage{hyperref}
\usepackage[affil-it]{authblk}
 
\begin{document}
\title{Mapping Matchings to Minimum Vertex Covers: K\H onig's Theorem Revisited}
\author[1]{Jacob Turner}
\affil[1]{Plotwise, Poortweg 4d, 2612 PA Delft, Netherlands.\\ \href{mailto:email@address}{jacob.turner870@gmail.com}}
\maketitle
\begin{abstract}
It is a celebrated result in early combinatorics that, in bipartite graphs, the size of maximum matching is equal to the size of a minimum vertex cover. K\H{o}nig's proof of this fact gave an algorithm for finding a minimum vertex cover from a maximum matching. In this paper, we revisit the connection this algorithm induces between the two types of structures. We find that all minimum vertex covers can be found by applying this algorithm to some matching and then classify which matchings give minimum vertex covers when this algorithm is applied to them.
\end{abstract}

\section{Introduction}

The minimum vertex cover problem is classical in graph theory and of the first problems to be classified as \textsf{NP}-complete \cite{karp1972reducibility} and is more generally a well studied problem in complexity theory \cite{dinur2005hardness, papadimitriou1998combinatorial, chen2006improved}. This problem also relates to maximum independent sets \cite{gallai1959uber}, dominating sets, and the Kuhn-Munkres (often known as Hungarian) algorithm in combinatorial optimization \cite{kuhn1955hungarian}.

In this paper we are interested in the strong connection between minimum vertex covers and maximum matchings in bipartite graphs, which has arisen in several fundamental results in combinatorics. The problem of finding the cardinality of a minimum vertex cover in a bipartite graphs was found to have a polynomial time algorithm when it was discovered independently by K\H{o}nig and Egerv\'ary that the size of maximum matching in a bipartite graph is in fact equal to the size of a minimum vertex cover \cite{koniggrafok,konig1916graphen, egervary1931matrixok}. As an optimization problem is a half-integral linear program that is the dual of the maximum matching problem \cite{vazirani2013approximation}.

In fact, the proof of K\H{o}nig's theorem supplies an algorithm for finding a minimum vertex cover from a maximum matching in bipartite graphs. This algorithm is still used in most state of the art algorithms today in order to find a minimum vertex cover after finding a maximum matching (which may be accomplished by one of several different algorithms \cite{hopcroft1973n, ford2009maximal, goldberg1988new}).

While the decision versions of these problems are solvable in polynomial time, counting and enumeration versions are intractable \cite{valiant1979complexity,provan1983complexity, lovasz2009matching}. Furthermore, in practical applications, the computation of maximum matchings is frequently a bottleneck in computing minimum vertex covers. An important example of this is the Kuhn-Munkres algorithm, as noted by Kuhn himself \cite{kuhn1956variants}.

This paper seeks to revisit the relationship between matchings and minimum vertex covers through the lens of the algorithm implicit in K\H{o}nig's proof of his eponymous theorem. A few natural questions stand out to us that we resolve in this paper:
\begin{itemize}
    \item If we are only interested in a maximum matching in order to compute a minimum vertex cover, can we make do with matchings satisfying a weaker condition?
    \item Can all minimum vertex covers in bipartite graphs be found by applying a uniform algorithm to matchings?
    \item Does the algorithm implicit to K\H{o}nig's proof  allow us to set up a correspondence between maximal matchings and minimum vertex covers?
\end{itemize}

The main results of this paper answer the first two questions in the affirmative and further classifies precisely which matchings map to minimum vertex covers. However, we will see that there is much subtlety to getting minimum vertex covers from matchings that ultimately prevents from making any stronger correspondence between maximal matchings and minimum vertex covers. We consider it an interesting question to determine for which graphs such a strong correspondence does exist and demonstrate an infinite family of graphs for which there is one.

The paper is organized as follows. Section \ref{sec:background} recalls the necessary terminology and results from graph theory. In Section \ref{sec:reverse}, we show that every minimum vertex cover can be found by applying the algorithm of K\H{o}nig to some matching. We then find an infinite family of bipartite graphs such that we may further assume the matchings are maximal. This family includes every bipartite graph as an induced subgraph of one of its members. In view this result, it is sufficient to classify which maximal matchings map to minimum vertex covers, which we do in Section \ref{sec:classification}.

\section{Background}\label{sec:background}

We first recall the terminology we will use before revisiting classical results and techniques from graph theory relating maximal matching and minimum vertex covers.

\begin{defn}
Let $G=(U,V,E)$ be a bipartite graph with independent sets $U$ and $V$. A \emph{matching} $\mathcal{M}$ is a subset of $E$ such that no two edges in $\mathcal{M}$ share an endpoint. We say that a vertex $v$ is \emph{saturated} by $\mathcal{M}$ if it is one of the endpoints of an edge in $\mathcal{M}$. We say a subset of vertices is saturated if every vertex in said subset is saturated.
\end{defn}

Vertex covers are a classically studied object in combinatorics and at first seem quite unrelated to maximal matchings. As it turns out, the two actually enjoy a close relationship.

\begin{defn}
Given a graph $G=(V,E)$, a vertex cover $\mathcal{C}$ is a subset of $V$ such that every edge in $E$ has at least one endpoint in $\mathcal{C}$.
\end{defn}

As discovered by K\H{o}nig and Egerv\'ary, maximum matchings can be mapped to minimum vertex covers in bipartite graphs and the procedure for doing this makes it clear that both sets have the same cardinality. To this day, essentially all state of the art algorithms for finding minimum vertex covers in bipartite graphs that does not use integer linear programming constructs a minimum vertex cover from a maximum matching.

\begin{theorem}[K\H{o}nig's Theorem \cite{konig1916graphen,koniggrafok}]\label{thm:konig}
For a bipartite graph, the size of maximum matching is equal to the size of a minimum vertex cover.
\end{theorem}

We first describe  the classical way to find a maximum matching in a bipartite graph $G=(U,V,E)$ as that is frequently the first step in computing a minimum vertex covering. Other algorithms for finding maximal matchings exists, for example push-relabel algorithms, but we do not need to understand them for this paper.

We begin by finding a maximal matching $\mcM$, which is easily done by selecting edges from $E$ to the matching $\mcM$ until the addition of any further edge violates the requirements of a matching. From there, we increase the size of $\mcM$ until it is of maximum cardinality.

\begin{defn}
Given a bipartite graph $G=(U,V,E)$ and a matching $\mcM$, we define a path $P=(w_1,\dots, w_n)$, $w_i\in U\cup V$, to be an \emph{alternating path} if every other edge in the path is in $\mcM$. Furthermore, if $n$ is even and neither $w_1$ nor $w_n$ is saturated by $\mcM$, we say that $P$ is an \emph{augmenting path}. Otherwise, we say that it is \emph{non-augmenting}.
\end{defn}

It is easy to check that for an augmenting path $P$ and a matching $\mathcal{M}$, the symmetric difference is again a matching. Therefore, if an augmenting path whose beginning and end edges are not in the matching is found, the matching cannot be maximal. In fact, the converse also holds \cite{berge1957two}.

 After the maximal matching $\mcM$ is found, for every unsaturated vertex $w$, a breadth first search is performed, keeping track of all alternating paths from $w$. If another unsaturated vertex is found, then the augmenting path $P$ is reconstructed and the matching $\mcM$ is replaced with the larger matching $\mcM\triangle P$, that is, the symmetric difference  of $\mcM$ and $P$. 

Once a maximum matching $\mcM$ is found, we can use the algorithm implicit in K\H{o}nig's proof of Theorem \ref{thm:konig} to compute a minimum vertex cover of $G$. If $G=(U,V,E)$, we assume that $|U|\le |V|$. We then construct the set $Z$ which consists of every vertex in $U\cup V$ reachable from an unsaturated vertex in $U$. Then the set $S:=(U\setminus Z)\cup(V\cap Z)$ is a minimum vertex cover.

\begin{defn}\label{def:zed-set}
Given a matching $\mcM$ of a bipartite graph $G$, let the vertex cover derived by performing the procedure of K\H{o}nig's in his proof of Theorem \ref{thm:konig} be denoted by $K_{\mcM}(G)$. We call the routine of computing $K_{\mcM}(G)$ \emph{K\H{o}nig's Procedure}. Let $Z_{\mcM}(G)$ be the set of all vertices reached by alternating paths from unsaturated vertices from the smaller independent set.
\end{defn}

A major question that lead to this paper is the following: If we are only interested in finding a minimum vertex cover, is it necessary to find a maximum matching? One quickly realizes that we don't.

\begin{example}\label{ex:p3}
Consider the following graph:
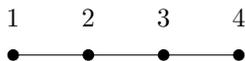
\vspace{5pt}
\begin{center}
\begin{tikzpicture}
\draw[fill=black] (-0.5,.5) circle (2pt);
\draw (-0.5, 1) node{$1$};
\draw (.5, 1) node{$2$};
\draw[fill=black] (.5,.5) circle (2pt);
\draw (-0.5,.5) -- (.5,.5);
\draw (.5,.5) -- (1.5,.5);
\draw (1.5, 1) node{$3$};
\draw[fill=black] (1.5,.5) circle (2pt);
\draw (2.5, 1) node{$4$};
\draw[fill=black] (2.5,.5) circle (2pt);
\draw (1.5,.5) -- (2.5,.5);
\end{tikzpicture}
\end{center}

The minimum vertex covers of this graph are $\{1,3\}$, $\{2,3\}$ and $\{2,4\}$, which is to say a choice of endpoint for each edge in the unique perfect matching, excepting $\{1,4\}$ which does not see edge $(2,3)$.

Two of these covers can be achieved by applying K\H{o}nig's procedure to the two maximal matchings this graph has: The perfect matching maps to $\{1,3\}$ and the matching with single edge $(2,3)$ maps to $\{2,4\}$. Intriguingly, the cover $\{2, 3\}$ can be found by applying K\H{o}nig's procedure to the matching with single edge $(3,4)$

\end{example}
\vspace{1cm}

Example \ref{ex:p3} shows that it is possible to find minimum vertex covers by widening the kinds of subgraphs of a bipartite graph that K\H{o}nig's procedure is applied to. Since, at a very high level, minimum vertex cover algorithms for bipartite graphs work by searching the space of subgraphs for one that K\H{o}nig's procedure may be applied to, we seek to understand this space better, especially as it is potentially much larger than just the set of maximum matchings. 

In this paper, we show that every minimum vertex cover can be found by applying K\H{o}nig's procedure to a matching and classify which matchings produce minimum vertex covers when this procedure is applied to them. Throughout the rest of this paper, we will assume that $G=(U,V, E)$ is a connected bipartite graph with bipartition given by the vertex sets $U$, $V$ and with edges $E$. Furthermore, we assume $|U|\le |V|$ unless stated otherwise. 

\section{Obtaining minimum vertex covers from matchings}\label{sec:reverse}

It is a natural question to ask if all minimum vertex covers come from applying K\H{o}nig's procedure to certain kinds of matchings. Amazingly, the answer is yes, although they do not all come from maximal matchings as we have seen from Example \ref{ex:p3}. We first recall a classic theorem from the theory of matchings and reprove a well-known corollary.

\begin{theorem}[Hall's Marriage Theorem \cite{hall2009representatives}]\label{thm:hall}
Let $G=(U, V, E)$ be bipartite. There is matching saturating every vertex in $U$ if and only if for every $W\subseteq U$, $|W|\le |\bigcup_{w\in W}{N(W)}|$.
\end{theorem}

\begin{lemma}\label{lem:one-sided-mvc}
Let $G=(U,V,E)$ be bipartite and suppose that $U$ is a minimum vertex cover. Then there exists a maximum matching of $G$ saturating $U$.
\end{lemma}
\begin{proof}
We first note that $|U|\le |V|$ as $U$ is a minimum vertex cover. Now let $W\subseteq U$ and consider the subgraph of $G$ induced by $W$ and the union of vertices $N(w)$ for $w\in W$. Call this graph $G(W)$. We claim that $W$ is a minimum vertex cover of $G(W)$.

First note that no vertex in $G\setminus G(W)$ is covered by an element of $W$. This means that every vertex in $G\setminus G(W)$ is covered by an element in $U\setminus W$. Now suppose that $G(W)$ has a minimum vertex cover smaller than $|W|$, call it $M$. Then $M\cup (U\setminus W)$ is a vertex cover of $G$ smaller than $W$, a contradiction.

Thus we have that $|W|\le |G(W)\cap V|$ since $W$ is a minimum vertex cover. But by the definition of $G(W)$, we get that $|W|\le |\bigcup_{w\in W}{N(w)}|$. Since $W$ was arbitrary, by Hall's Marriage Theorem (\ref{thm:hall}), we get the desired result. The fact the matching is maximum is trivial.
\end{proof}

Given a vertex cover $\mcC$ of a bipartite graph $G=(U,V,E)$, it is easy to see that those edges which have both endpoints in $\mcC$ form a cut of $G$. This follows directly the definition of a vertex cover. 

\begin{defn}
Given a bipartite graph $G$ and a minimum vertex cover $\mcC$, define $G_{\uparrow}\langle \mcC\rangle$ as the subgraph of $G$ induced by the vertices $(V\cap \mcC)\cup(U\setminus\mcC)$. Define $G_{\downarrow}\langle \mcC\rangle$ as the subgraph induced by $(U\cap \mcC)\cup(V\setminus\mcC)$. 
\end{defn}

We note that the cut obtained from a $\mcC$ as described above is precisely the cut that produces $G_{\uparrow}\langle \mcC\rangle$ and $G_{\downarrow}\langle \mcC\rangle$. Furthermore, by Lemma \ref{lem:one-sided-mvc}, both of these subgraphs have maximum matchings saturating the smaller of their two partitions. The union of these two maximum matchings produces a maximum matching of the whole graph, giving another proof of the K\H{o}nig's Theorem.

However, on the graph $G_{\uparrow}\langle\mcC\rangle$, we are interested in a matching potentially different from the one guaranteed by Lemma \ref{lem:one-sided-mvc}. In fact, we will try to produce a matching by applying K\H{o}nig's procedure in reverse.

Given a minimum vertex cover $\mcC$ of a bipartite graph $G$, we produce a matching $\mcM_{\mcC}$ on $G_{\uparrow}\langle\mcC\rangle$ via the following procedure.

\vspace{.2cm}

\vspace{.1cm}
\noindent\textbf{Procedure:} $Subroutine(G_{\uparrow}\langle\mcC\rangle, u,\mcM)$
\hrule
\vspace{.2cm}
    \begin{algorithmic}
    \For{all unsaturated neighbors $v$ of $u$}
    \For{all unsaturated neighbors $w$ of $v$} 
    \State{add $(v,w)$ to $\mcM$}
    \State{$\mcM\gets \mcM\cup RK(G_{\uparrow}\langle\mcC\rangle, v,\mcM)$}
    \EndFor
    \EndFor
    \Return{$\mcM$}
    \end{algorithmic}
\hrule
\newpage
\vspace{.2cm}
\noindent\textbf{Procedure:} $Main(G_{\uparrow}\langle\mcC\rangle)$\label{alg:rev_konig}
\hrule
\vspace{.2cm}
    \begin{algorithmic}
    \State{$\mcM\gets\{\}$}
    \For{$u$ in $U\setminus\mcC$}
    \If{$u$ is unsaturated in $\mcM$}
    \State{$\mcM\gets \mcM\cup Subroutine(G_{\uparrow}\langle\mcC\rangle, u,\mcM)$}
    \EndIf
    \EndFor
    \Return{$\mcM$}
    \end{algorithmic}
\hrule
\vspace{.2cm}

It is clear the subgraph produced by Procedure \ref{alg:rev_konig} is a matching given that it is checked that both endpoints of any edge added to the returned subgraph are both unsaturated. We further claim that applying K\H{o}nig's procedure to this matching gives us back the original matching on $G_{\uparrow}\langle\mcC\rangle$.

\begin{theorem}\label{thm:reverse-konig}
Let $\mcC$ be a minimum vertex cover of a bipartite graph $G$. Let $\mcM_{\uparrow}$ be the matching produced by Procedure \ref{alg:rev_konig} on $G_{\uparrow}\langle\mcC\rangle$ and let $\mcM_{\downarrow}$ be the matching on $G_{\downarrow}\langle\mcC\rangle$ guaranteed by Lemma \ref{lem:one-sided-mvc}. Then applying K\H{o}nig's procedure to the matching $\mcM:=\mcM_{\uparrow}\cup\mcM_{\downarrow}$ produces $\mcC$.
\end{theorem}
\begin{proof}
First of all, since every vertex in $G_{\downarrow}\langle\mcC\rangle$ is saturated by Lemma \ref{lem:one-sided-mvc} and none of the edges in the cut between $G_{\uparrow}\langle\mcC\rangle$ and $G_{\downarrow}\langle\mcC\rangle$ are in $\mcM$, none of the vertices in $G_{\downarrow}\langle\mcC\rangle$ are reachable via alternating paths from unsaturated vertices in $G$. Thus $K_{\mcM}(G)$ agrees with $\mcC$ on $G_{\downarrow}\langle\mcC\rangle$.

We now need to show that $K_{\mcM}(G)$ agrees with $\mcC$ on $G_{\uparrow}\langle\mcC\rangle$. First let us show that every vertex in $U\setminus\mcC$ is reachable from some alternating path starting from an unsaturated vertex in $U\setminus\mcC$. Suppose some vertex $w$ were not. Then by the Procedure \ref{alg:rev_konig}, we have that $w$ is unsaturated and thus reachable trivially. 

Now suppose some $v\in V\cap\mcC$ is not reachable by an alternating path starting from an unsaturated vertex in $U\setminus \mcC$. This firstly implies that $v$ is unsaturated and secondly implies that none of its neighbors in $U\setminus\mcC$ were found and saturated either. 

If $v$ has any neighbors in $G_{\uparrow}\langle \mcC\rangle$, they are unsaturated and thus $v$ is reachable after all, ensuring $v\in K_{\mcM}(G)$. Now suppose $v$ has no neighbors in $G_{\uparrow}\langle\mcC\rangle$. Then, since $G$ is connected, it has only neighbors in $G_{\downarrow}\langle\mcC\rangle$ implying that $v$ and its entire neighborhood is in $\mcC$. This contradicts that assumption that $\mcC$ is minimum. So in sum $K_{\mcM}(G)$ agrees with $\mcC$ on $G_{\uparrow}\langle G\rangle$, proving the theorem.
\end{proof} 

By Theorem \ref{thm:reverse-konig}, Procedure \ref{alg:rev_konig} gives an algorithm that acts as a sort of inverse inverse to K\H{o}nig's procedure. However, it may well produce difference results if the vertices in $U\setminus\mcC$ are visited in a different order. Indeed, we expect this to often be the case as the following example and proposition shows that  K\H{o}nig's procedure is not injective.

\begin{example}\label{ex:non-injective}
Consider the following graph and matching (filled lines means the edge is in the matching):
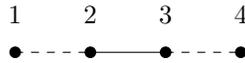
\vspace{5pt}
\begin{center}
\begin{tikzpicture}
\draw[fill=black] (-0.5,.5) circle (2pt);
\draw (-0.5, 1) node{$1$};
\draw (.5, 1) node{$2$};
\draw[fill=black] (.5,.5) circle (2pt);
\draw[dashed] (-0.5,.5) -- (.5,.5);
\draw (.5,.5) -- (1.5,.5);
\draw (1.5, 1) node{$3$};
\draw[fill=black] (1.5,.5) circle (2pt);
\draw (2.5, 1) node{$4$};
\draw[fill=black] (2.5,.5) circle (2pt);
\draw[dashed] (1.5,.5) -- (2.5,.5);
\end{tikzpicture}

\end{center}
If we let $U=\{1,3\}$ and $V=\{2,4\}$, then we see that the empty matching and the above matching both yield the minimum vertex covering $\{2, 4\}$
\end{example}

\begin{proposition}\label{prop:cycle-fiber}
Let $G$ be bipartite and $\mcM$ and  $\mcM'$ be two matchings such that $\mcM\triangle \mcM'$ is disjoint union of cycles. Then $K_{\mcM}(G)=K_{\mcM'}(G)$.
\end{proposition}
\begin{proof}
We first note that the vertices of $U$ saturated by $\mcM$ and $\mcM'$ are the same. We is enough to show that $Z_{\mcM}(G)=Z_{\mcM'}(G)$. Let $C=\mcM\triangle\mcM'$.

Let $u\in U$ be unsaturated and suppose there is some alternating path $p$ from $u$ to $v$ with respect to $\mcM$, i.e. $v\in Z_{\mcM}(G)$. If $p\cap C=\emptyset$, then $v\in Z_{\mcM'}(G)$ trivially. So suppose that $p\cap C\ne\emptyset$.

We assume at first that $p$ intersects a single cycle in $C$. From this case it will be clear how to handle the case that $p$ intersects several cycles in $C$.
In $\mcM$ we may write $p$ as the concatenation of paths $p_0\to p_1\to p_2$ where $p_1=p\cap C$. Now let $\hat{p}_1$ be the path with the same endpoints as $p_1$ but traversing the cycle of $C$ that $p$ intersects in the other way. Then $p_0\to \hat{p}_1\to p_1$ is an alternating path in $\mcM'$.

So again $v\in Z_{\mcM'}(G)$ and it is clear how to extend this argument if $p$ intersect multiple cycles of $C$. Furthermore, all the above arguments work to show that for $v\in Z_{\mcM'}(G)$, we a have that $v\in Z_{\mcM}(G)$.
\end{proof}

\subsection{Enumeratively K\H{o}nig-Egerv\'ary Graphs}

We have seen that not every minimum matching can be obtained by applying K\H{o}nig's procedure to a maximal matching, perhaps contrary to expectation. We find it an interesting question for which graphs can we find all minimum vertex covers from maximal matchings.

\begin{defn}
A bipartite graph is \emph{enumeratively K\H{o}nig-Egerv\'ary} if every minimum vertex cover can be found by applying K\H{o}nig's procedure to a maximal matching.
\end{defn}

Example \ref{ex:p3} shows that not every bipartite graph is enumeratively K\H{o}nig-Egerv\'ary. Furthermore, we do not know a classification of such graphs, and so we leave it as an open question.

However, we do show that are certain set of bipartite graphs are enumeratively K\H{o}nig-Egerv\'ary. Furthermore, this set is very large in the sense that every bipartite graph can be found as a induced subgraph of a graph in this family. As such, when we turn our attention to the problem of classifying which matchings yield minimum vertex covers when K\H{o}nig's procedure is applied, it is really enough to focus our attention on maximal matchings.

\begin{defn}
 Let $3\star$ denote the star graph with three leaves. Given a bipartite graph $H$, let $St(H)$ be the graph formed by attaching a copy of $3\star$ to every vertex of $H$. We call a bipartite graph $G$ \emph{star-studded} if it is of the form $St(H)$ for some bipartite graph $H$.
\end{defn}

We make a couple of important observations about star-studded graphs. First of all, they are clearly bipartite. Secondly, it is a  quick exercise to check if $H=(U,V, E)$ with $|U|\le |V|$ and $St(H)=(X, Y, E')$ with $|X|\le |Y|$, then $U\subset X$ and $V\subset Y$. Thus we apply K\H{o}nig's procedure from the ``same side" in both $H$ and $St(H)$.

We now investigate the conditions under which Procedure \ref{alg:rev_konig} might produce a matching that is not maximal. We know that the matching produced on $G_{\downarrow}\langle\mcC\rangle$ is maximal, so we focus our attention on $G_{\uparrow}\langle\mcC\rangle$.

Throughout this section, $\mcM_{G,\mcC}$ will refer to matching on $G$ found by applying Procedure \ref{alg:rev_konig} to a minimum vertex cover $\mcC$ of $G$. We refer to this matching by the same name even if we are focusing on a subgraph of $G$.

Suppose there are two neighboring vertices $u,v$ in $G_{\downarrow}\langle\mcC\rangle$ that are unsaturated by $\mcM_{G,\mcC}$. Since we cannot add the edge $(u,v)$ to the matching and maintain the property that $K_{\mcM}(G)=\mcC$, we must have that $u$ is the only unsaturated vertex in $U$ from which $v$ can be reached. This further implies that $u$ itself cannot be reached via an alternating path from another unsaturated vertex in $U\setminus\mcC$, else it would be saturated

\begin{lemma}\label{lem:lone-soldier}
If neighbors $u, v\in G_{\uparrow}\langle\mcC\rangle$ are both unsaturated by $\mcM_{G,\mcC}$, then $u$ is the only neighbor $v$ in $G_{\uparrow}\langle\mcC\rangle$.
\end{lemma}
\begin{proof}
Suppose $v$ had another neighbor in $G_{\uparrow}\langle\mcC\rangle$, $u'$. First suppose that $u'$ is unsaturated. Then the edge $(u',v)$ or $(u,v)$ would have been added to $\mcM_{G,\mcC}$ by Procedure \ref{alg:rev_konig} if it had visited $u$ or $u'$ first in the main loop, respectively. This contradicts our assumption that $v$ is unsaturated.

Secondly, suppose that $u'$ is saturated. Then it is reachable via some alternating path from an  unsaturated vertex $u''\in U\setminus \mcC$. But then so is $v$ and $u$ and thus Procedure \ref{alg:rev_konig} would have added $(u, v)$ to $\mcM_{G,\mcC}$, another contradiction.
\end{proof}

As a consequence of Lemma \ref{lem:lone-soldier}, it may be that for some of the $u\in U\setminus \mcC$ we choose for the main loop of Procedure \ref{alg:rev_konig}, these $u$ may have a neighbor whose only other neighbors are in $G_{\downarrow}\langle\mcC\rangle$.

\begin{lemma}\label{lem:mvc-bijection}
If $G=St(H)$  there is a bijection between minimum vertex covers of $H$ and minimum vertex covers of $St(H)$.
\end{lemma}
\begin{proof}

Since $G$ is star-studded, we may consider it as $St(H)$ for some bipartite graph $H$. Since $H$ is an induced subgraph, every minimum vertex cover of $G$ restricts to a minimum vertex cover on $H$. We then note that any minimum vertex cover of $G$ includes the centers of all the copies of $3\star$ that we added to $H$ and none of the leaves of $G$. Thus every minimum vertex cover of $H$ lifts uniquely to a minimum vertex cover of $G$.

\end{proof}

\begin{theorem}\label{thm:stars}
Star-studded graphs are enumeratively K\H{o}nig-Egerv\'ary. Furthermore, for a bipartite graph $H$, every minimum vertex cover of $H$ can be found by applying K\H{o}nig's procedure to a maximal matching matching in $St(H)$ and restricting the resulting cover to $H$.
\end{theorem}
\begin{proof}
The second part of the theorem  will follow directly from the first part and Lemma \ref{lem:mvc-bijection}. So we need only focus on the first part.

Let $G=St(H)=(U,V,E)$ be star-studded and let $\mcC$ be a minimum vertex cover of $G$ such that Proposition \ref{alg:rev_konig} may produce a non-maximal matching $\mcM$. By Lemma \ref{lem:lone-soldier}, we know that if neighbors $u,v\in G_{\uparrow}\langle \mcC\rangle$ are both unsaturated then $u$ is the unique neighbor of $v$ in $G_{\uparrow}\langle\mcC\rangle$.

We will only change $\mcM$ on the copies of $3\star$ attached to $H$ in order to saturate $u$ and doing it in such a way that does not change the outcome of K\H{o}nig's procedure.

 Suppose $v$ is in $H$. The $3\star$ attached to $v$ has its center in $\mcC$. So $u$ must also be in $H$. Thus we change $\mcM$ on the $3\star$ attached to $u$ to include the edge between $u$ and the center of the $3\star$. We see that this does not change the minimum vertex cover produced by K\H{o}nig's procedure as $u$ is now reachable via an alternating path from the leaves of the $3\star$ attached to it.
 
 Now suppose $v$ is not in $H$, thus it is in one of the attached copies of $3\star$. We do not need to consider the case that the copy of $3\star$ is attached to a vertex in $V$ as that would imply that $v\in G_{\downarrow}\langle\mcC\rangle$ and we know that $\mcM$ restricted to this subgraph is maximal by Lemma \ref{lem:one-sided-mvc}.
 
So suppose that the $3\star$ is attached to an element of $u\in U$. So $v$ is the center of the $3\star$. If $u$ is unsaturated, we change the matching to include $(u,v)$ as in the previous analysis. If $u$ is saturated, we make sure that $v$ is matched with one of the two leafs incident to it in the $3\star$. We see that $u$ is not reachable by alternating path from either of these two leaves so this will not change whether or not $u$ is in $\mcC$.
 
 We see that in this way, we can increase the matching until it is maximal with affecting the minimum vertex cover produced by K\H{o}nig's procedure.
\end{proof}

One the major upshots of Theorem \ref{thm:stars} is that for the rest of the paper, we may restrict our attention to maximal matchings for convenience. Indeed, in the next section we investigate the question of which maximal matchings produce minimum vertex covers. The resulting characterization can be easily extended from maximal matchings to all matchings using Theorem \ref{thm:stars}.

\section{K\H{o}nig's Procedure Applied to Maximal Matchings}\label{sec:classification}

In the previous sections, we have seen that every minimum vertex cover can be obtained by applying K\H{o}nig's procedure to a matching. In fact, we have even described an algorithm to recover matchings from minimum vertex covers and shown that this algorithm is an inverse of sorts to K\H{o}nig's procedure.

However, we still have little idea of what kinds of matchings this reverse procedure can produce. To put it another way, what criteria must a matching satisfy in order to map to a minimum vertex cover via K\H{o}nig's procedure? In this section  we classify those maximal matchings such that $K_{\mcM}(G)$ is a minimum vertex cover. We use Theorem \ref{thm:stars} to justify our focus on maximal matchings rather than arbitrary matchings. 

\begin{lemma}\label{lem:one-endpoint}
If $\mcM$ is matching of a bipartite graph $G=(U,V,E)$, then for every $e \in \mcM$, exactly one endpoint is in $K_{\mcM}(G)$
\end{lemma}
\begin{proof}
 For every $e\in \mcM$, either both endpoints are in $Z_{\mcM}(G)$ or neither are, by construction. Since $K_{\mcM}(G) = (U\setminus Z_{\mcM}(G))\cup(V\cap Z_{\mcM}(G))$, this proves the lemma.
\end{proof}

\begin{proposition}\label{prop:max-min}
If $\mcM$ is a maximal matching of a bipartite graph $G=(U,V,E)$, then $K_{\mcM}(G)$ is a minimal vertex cover.
\end{proposition}
\begin{proof}
 Let $e$ be an edge with endpoint $e_u\in U$ and $e_v\in V$. Because $\mcM$ is maximal, there is no edge $e$ such that both endpoints are not saturated. We consider the following cases:
\begin{enumerate}
    \item $e_u$ is unsaturated.
    \item $e_u$ is saturated but $e_v$ is unsaturated.
    \item Both $e_u$ and $e_v$ are saturated.
\end{enumerate}

If $e_u$ is unsaturated, then $e_v\in K_{\mcM}(G)$ and $e$ is covered. Suppose $e_v$ is unsaturated but $e_u$ is saturated. If $e_u\in K_{\mcM}(G)$, then $e$ is covered. Else, if $e_u\notin K_{\mcM}(G)$, then it is reachable by an alternating path from an unsaturated vertex in $U$, thus so is $e_v$ implying $e_v\in K_{\mcM}(G)$. Thus $e$ is still covered.

The last case to consider is (3). If $e\in \mcM$, then $e$ is covered by Lemma \ref{lem:one-endpoint}. Suppose $e\notin \mcM$ and $e_u\notin K_{\mcM}(G)$. Let $d=(e_u, d_v)\in \mcM$ be the edge saturating $e_u$. Thus $d_v\in K_{\mcM}(G)$ meaning it is reachable by an alternating path from an unsaturated vertex in $U$. This implies that $e_u$ and $e_v$ are also reachable by alternating paths from an unsaturated vertices in $U$ and thus $e_v\in K_{\mcM}(G)$ and so $e$ is covered. Otherwise, $e_u\in K_{\mcM}(G)$ and thus $e$ is still covered.
 
 Now we wish to show that $K_{\mcM}(G)$ is minimal. Let suppose otherwise, which implies there is a vertex $r\in K_{\mcM}(G)$ such that $N(r)\subseteq K_{\mcM}(G)$, where $N(r)$ is the set of neighbors or $r$.
 
 First suppose that $r\in U$. It must be that $r$ is saturated as $r\in K_{\mcM}(G)$, which implies that the vertex matched with $r$ is not in $K_{\mcM}(G)$ by Lemma \ref{lem:one-endpoint}. So we can conclude that $r\in V$. Lemma \ref{lem:one-endpoint} again implies that $r$ cannot be saturated. So there exists an augmenting path $P$ from some unsaturated vertex $u\in U$ to $r$. Let $u'\in U$ be the next to last vertex in $P$ which is in $N(r)$. By the construction of $K_{\mcM}(G)$, $u'\notin K_{\mcM}(G)$ which contradicts our assumption. So $K_{\mcM}(G)$ is minimal.
 
\end{proof}

While Proposition \ref{prop:max-min} guarantees that $K_{\mcM}(G)$ gives a vertex cover, it may not be minimum if $\mcM$ is not maximum. The following example illustrates this.

\begin{example}\label{ex:notminimal}
Consider the following graph with the dashed edges representing the edges not in the matching and the filled edges representing a maximal matching.

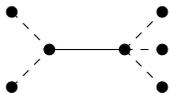
\vspace{5pt}
\begin{center}
\begin{tikzpicture}
\draw[fill=black] (0,0) circle (2pt);
\draw[fill=black] (0, 1) circle (2pt);
\draw[fill=black] (.5,.5) circle (2pt);
\draw[dashed] (0,0) -- (.5,.5);
\draw[dashed] (0,1) -- (.5,.5);
\draw (.5,.5) -- (1.5,.5);
\draw[fill=black] (2, 0) circle (2pt);
\draw[fill=black] (2,.5) circle (2pt);
\draw[fill=black] (2,1) circle (2pt);
\draw[fill=black] (1.5,.5) circle (2pt);
\draw[dashed] (1.5,.5) -- (2,.5);
\draw[dashed] (1.5,.5) -- (2,0);
\draw[dashed] (1.5,.5) -- (2,1);
\end{tikzpicture}
\end{center}

Here $|U|=3$ and $|V|=4$. Applying K\H{o}nig's Procedure yields the following vertex cover, where filled vertices belong to the cover and hollow vertices do not.

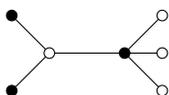
\vspace{5pt}
\begin{center}
\begin{tikzpicture}
\draw[fill=black] (0,0) circle (2pt);
\draw[fill=black] (0, 1) circle (2pt);
\draw (0,0) -- (.5,.5);
\draw (0,1) -- (.5,.5);
\draw (.5,.5) -- (1.5,.5);
\draw[fill=black] (1.5,.5) circle (2pt);
\draw (1.5,.5) -- (2,.5);
\draw (1.5,.5) -- (2,0);
\draw (1.5,.5) -- (2,1);
\draw [fill=white] (.5,.5) circle (2pt);
\draw [fill=white] (2, 0) circle (2pt);
\draw[fill=white] (2,.5) circle (2pt);
\draw[fill=white] (2,1) circle (2pt);
\end{tikzpicture}
\end{center}

We see that this is indeed a vertex cover, with cardinality four. However, there is clearly a smaller vertex cover of cardinality two, namely the following.
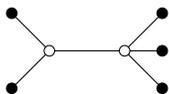
\vspace{5pt}
\begin{center}
\begin{tikzpicture}
\draw[fill=black] (0,0) circle (2pt);
\draw[fill=black] (0, 1) circle (2pt);
\draw (0,0) -- (.5,.5);
\draw (0,1) -- (.5,.5);
\draw (.5,.5) -- (1.5,.5);
\draw[fill=black] (2, 0) circle (2pt);
\draw[fill=black] (2,.5) circle (2pt);
\draw[fill=black] (2,1) circle (2pt);
\draw (1.5,.5) -- (2,.5);
\draw (1.5,.5) -- (2,0);
\draw (1.5,.5) -- (2,1);
\draw[fill=white] (.5,.5) circle (2pt);
\draw[fill=white] (1.5,.5) circle (2pt);
\end{tikzpicture}
\end{center}
\end{example}

Example \ref{ex:notminimal} clearly shows that not all maximal matchings produce minimum vertex covers when applying K\H onig's procedure. As it turns out, the above is an example of a more general phenomenon which characterize all maximal matchings that do not produce minimum vertex covers. This will be shown in Section \ref{sec:classification} by  Proposition \ref{prop:counter_example}.

\subsection{Classifying which maximal matchings map to minimum vertex covers}
 We approach the classification problem by trying to understand the difference in the cardinalities of $K_{\mcM}(G)$ versus $K_{\mcM\triangle P}(G)$ where $P$ is an augmenting path with respect to $\mcM$. It turns out the difference is quite local in the graph.

\begin{defn}
Let $G=(U, V, E)$ be a bipartite graph, $\mcM$ a matching, and $P$ an augmenting path from some unsaturated vertex $u\in U$ to some $v\in V$. Define $G_{\mcM}\langle P\rangle$ be the subgraph of $G$ of all augmenting paths starting from some unsaturated vertex in $U$ that edge-wise intersect $P$.
\end{defn}

\begin{defn}
Given a matching $\mcM$ and a bipartite graph $G=(U,V,E)$, if $H=(V_H,E_H)$ is a subgraph then define 
\begin{itemize}
    \item $K_{\mcM}(H):=K_{\mcM}(G)\cap V_H,$
\item$ Z_{\mcM}(H):=Z_{\mcM}(G)\cap V_H.$
\end{itemize}

\end{defn}
\begin{observation}\label{obs:decomposition}
Given an augmenting path $P$, Let $P_1=P, \dots, P_k$ be all augmenting paths from some unsaturated vertex in $U$ that edge-wise intersect $P$. Then
$$G_{\mcM}\langle P\rangle=\bigcup_{i=1}^{n}{P_i},$$
$$K_{\mcM}(G_{\mcM}\langle P\rangle) = \bigcup_{i=1}^n{K_{\mcM}(P_i)}.$$
\end{observation}

\begin{observation}\label{obs:edge-means-vertex}
If two alternating paths intersect vertex-wise, they either intersect edge-wise or intersect only at their endpoints.
\end{observation}

We claim that it is enough to study the cardinalities of sets of the form $K_{\mcM}(G_{\mcM}\langle P\rangle)$ to understand when $K_{\mcM}(G)$ gives a minimum vertex cover.

\begin{theorem}\label{thm:technical_theorem}
Suppose $r$ is a vertex in $G$ that is not in $G_{\mcM}\langle P\rangle$ for a given augmenting path $P$, then $r\in K_{\mcM}(G)$ if and only if $r$ or the vertex it is matched to in $\mcM$ is in $K_{\mcM\triangle P}(G)$.
\end{theorem}
\begin{proof}
We first note that $r$ cannot be an unsaturated vertex of $U$. Now
there are two cases we consider: the first is that $r$ is reachable by an alternating path from some unsaturated vertex $u'\in U$. Let $P$ be an augmenting path from $u$ to $v$. The assumptions of the theorem implies that $r\ne u,v$. We consider two further subcases. First suppose that $u'\ne u$ and let $Q$ be the alternating path $Q$ from $u'$ to $r$.  By assumption and Observation \ref{obs:edge-means-vertex}, $Q$ does not vertex-wise intersect $P$. Therefore, $Q$ is an alternating path from an unsaturated vertex in $u'$ to $r$ in $\mcM\triangle P$. Thus $r\in K_{\mcM\triangle P}(G)$ if and only if $r\in K_{\mcM}(G)$.

Now suppose every alternating path from some unsaturated vertex in $U$ to $r$ must begin at $u$. Then in $\mcM\triangle P$ there is no alternating path from an unsaturated vertex in $U$ to $r$. In this case $r\in K_{\mcM\triangle P}(G)$ if and only if the vertex matched to $r$ in $\mcM$ is in $K_{\mcM}(G)$. We claim that $r$ is matched in $\mcM$. Note that if $r\in V$, then if $r$ is unsaturated, it violates our assumptions that $r\notin G_{\mcM}\langle Q\rangle$ for any augmenting path $Q$. If $r\in U$, then $r$ unsaturated implies we are in the case handled above.

Now suppose that $r$ is not reachable by any alternating path from an unsaturated vertex in $U$ with respect to $\mcM$ and $P$ be as before. We wish to show this remains true with respect to $\mcM\triangle P$ as well. Let us suppose the opposite. Let $Q$ be the alternating path to $r$ from some unsaturated $u'\in U$ with respect to $\mcM\triangle P$.

We clearly have that $P$ and $Q$ edge-wise intersect since $r$ is not reachable by any alternating path from an unsaturated vertex in $U$ with respect to $\mcM$.
 The vertices of $P$ are considered ordered from $u$ to $v$ and the vertices of $Q$ ordered from $u'$ to $r$. Furthermore, we say that the \emph{higher end} of an edge in $P$ or $Q$ is the greater vertex with respect to this order. We likewise define the \emph{lower end}. There are to cases to consider: the first is when the ordering of vertices of $Q$ agrees with the vertex ordering of vertices of $P$ on their intersection and the second where the ordering on the intersection is reversed.
 
 With respect to the vertex ordering of $P$, let $\iota_1$ be the first vertex where $P$ and $Q$ intersect and $\iota_2$ be the second. Since $Q$ is an alternating path in $\mcM\triangle P$, this implies that the edge of $P$ whose higher end is $\iota_1$ is in $\mcM$. Similarly, the edge of $P$ whose lower end is $\iota_2$ is also in $\mcM$.
 
 Now suppose that the vertex orderings of $P$ and $Q$ agree on their intersection. This implies that the concatenation of the path from $u$ to $\iota_1$ along $P$ and the path from $\iota_1$ to $u'$ along $Q$ (in the reverse direction) is an alternating path with respect to $\mcM$ connecting two unsaturated vertices of $U$. This is clearly impossible. 
 
 If the vertex orderings of $P$ and $Q$ do not agree on their intersection, then the path from $u$ to $\iota_1$ via $P$ concatenated with the path from $\iota_1$ to $r$ along $Q$ (in the reverse direction), gives an alternating path from $u$ to $r$ edge-wise intersecting $P$, contradicting our assumption that $r\notin G_{\mcM}\langle P\rangle$. 
 
 Thus $r$ is not reachable via any alternating path from an unsaturated vertex in $U$ with respect to $\mcM\triangle P$ and so $r\in K_{\mcM}(G)$ if and only if $r\in K_{\mcM\triangle P}(G)$. This concludes the proof
\end{proof}
\begin{corollary}\label{thm:sufficient-for-mvc}
Let $G=(U,V,E)$ be a bipartite graph and $\mcM$ be a matching. Then for any augmenting path $P$, $$|K_{\mcM}(G_{\mcM}\langle P\rangle)| = |K_{\mcM\triangle P}(G_{\mcM}\langle P\rangle)| \implies |K_{\mcM}(G)| = |K_{\mcM\triangle P}(G)|.$$ In particular, if $\mcM$ is maximal, then $K_{\mcM}(G)$ is a minimum vertex cover if for every augmenting path $P$, $$|K_{\mcM}(G_{\mcM}\langle P\rangle)|=|K_{\mcM\triangle P}(G_{\mcM}\langle P\rangle)|.$$
\end{corollary}
\begin{proof}
 Theorem \ref{thm:technical_theorem} plus our assumptions imply that $|K_{\mcM}(G)|=|K_{\mcM\triangle P}(G)|$ by Lemma \ref{lem:one-endpoint}. If $\mcM$ is maximal, we may proceed inductively by extending augmenting paths until we get a maximum matching $\tilde{\mcM}$. Proposition \ref{prop:max-min} then implies that $K_{\mcM}(G)$ is a minimum vertex cover.
\end{proof}

Theorem \ref{thm:technical_theorem} and Corollary \ref{thm:sufficient-for-mvc} give us the tools we need to investigate how we may extend the number of maximal matchings we can use to form a minimum vertex cover.

\begin{proposition}\label{prop:unique-ends}
Let $\mcM$ be a maximal matching of a bipartite graph $G=(U,V,E)$ with vertices $u\in U$ and $v\in V$ connected by an augmenting path $P$. Suppose that $u$ is the only unsaturated vertex of $U$ in $G_{\mcM}\langle P\rangle$. Then $K_{\mcM}(G)=K_{\mcM\triangle P}(G).$
\end{proposition}
\begin{proof}
Let $S=G_{\mcM}\langle P\rangle $. Using Corollary \ref{thm:sufficient-for-mvc}, it suffices to prove that $|K_{\mcM}(S)| = |K_{\mcM\triangle P}(S)|$.
Let $P=P_1, \dots, P_n$ be the augmenting paths from $u$ to $v$ in no particular order. We note that 
$$S = \bigcup_{i=1}^n{P_i}.$$
Further, let $\tilde{\mcM}:=\mcM\triangle P_1$. Note that with respect to $\tilde{\mcM}$, the paths $P_2,\dots, P_n$ are no longer alternating.

Note that for $\mcM$, the set of all vertices reachable by alternating paths from unsaturated vertices in $U$ contains every vertex in $P_i$ for all $i$. This is not the case in $\tilde{\mcM}$. The vertices in $P_i\setminus P_1$ for $2\le i\le n $ are not reachable by any alternating path in $\tilde{\mcM}$ from an unsaturated vertex in $U$ by the assumptions of the proposition. This is true even if there is an augmenting path from $u'\ne u$ to $v$ in $\mcM$.

Let $H$ be the union of all of the graphs $P_i\setminus P$ for $2\le i\le n$. Note that $H$ is a subgraph of $S$. For $2\le i\le n$, note that $P_i\setminus P$ is a disjoint union of paths with an odd number of edges.  Let one connected component be the path with vertices $w_1,\dots, w_k$ where $k$ is clearly even. The vertices $w_j$ with $j$ even are included in $K_{\mcM}(S)$. However in $\tilde{\mcM}$, since none of these vertices are reachable by alternating paths from unsaturated vertices in $U$, the reverse happens. And so $w_j$ for $j$ odd are in $K_{\tilde{\mcM}}(S)$. Since $k$ is even, 
$$|K_{\mcM}(P_i\setminus P)| = |K_{\tilde{\mcM}}(P_i\setminus P)|$$ for all $2\le i\le n$ implying that $|K_{\mcM}(H)|=|K_{\tilde{\mcM}}(H)|$. Lastly, note that $K_{\mcM}(P)=P\cap V$ and $K_{\tilde{\mcM}}(P) = P\cap U$. Since $P$ has an even number of vertices, we get that $|K_{\mcM}(P)|=|K_{\tilde{\mcM}}(P)|$. Since $S=H\cup P$, we get that $|K_{\mcM}(S)|=|K_{\tilde{\mcM}}(S)|$, proving the proposition.
\end{proof}

\begin{defn}
Given an alternating path from some $u$ to $v$, call it $P$, we define a total order on $P$, $\le_P$, as the order of vertices along $P$ starting from $u$ and ending at $v$. We call this the $P$-induced order.
\end{defn}
\begin{defn} Let $P$ and $Q$ be two alternating paths in a graph with their induced orders. Let $P\vee Q=\min\{x\;|\;\; x\in P\cap Q\}$ be the smallest vertex, with respect to $\le_P$, where the paths intersect, if it exists. Let $P\wedge Q$ be the largest vertex, with respect to $\le_P$ where the paths intersect, if it exists.
\end{defn}

One potentially problematic point of the above definition is the following: given two alternating paths $P$ and $Q$ that edge-wise intersect, $\le_P$ and $\le_Q$ may be the reverse of each other on their intersection. Thus, we do not have that $\vee$ and $\wedge$ are commutative operations, a priori. As it turns out, however, these operations are commutative as the above detailed situation cannot occur. 

\begin{lemma}
Given two augmenting paths $P, Q$ that intersect, $P\vee Q= Q\vee P$ and $P\wedge Q=Q\wedge P$.
\end{lemma}
\begin{proof}
It suffices to show that $\le_P,\le_Q$ agree on $P\cap Q$. If $P$ and $Q$ do not edge-wise intersect, this is trivial.

So let us assume that $P$ and $Q$ edge-wise intersect and assume for contradiction's sake that $\le_P$ is the reverse of $\le_Q$ on $P\cap Q$. Let $P$ connect $u$ and $v$ and $Q$ connect $u'$ and $v'$. Then we note that that path from $u$ to $P\wedge Q$ along $P$ and then on to $u'$ along $Q$ is an augmenting path between two elements in $U$, a contradiction. We can similarly get an augmenting path between $v$ and $v'$ which is equally impossible. Thus it must be that $\le_P$ and $\le_Q$ agree on their intersection.
\end{proof}

In light of Proposition \ref{prop:unique-ends}, we now define two related and useful constructions. Given a bipartite graph $G=(U,V,E)$ and a matching $\mcM$, let $P$ be an augmenting path from $u$ to $v$. Let $u=u_1,\dots,u_n$, $n>1$ be the unsaturated vertices of $U$ in $G_{\mcM}\langle P\rangle$ and $P=P_1,\dots, P_n$ be the corresponding augmenting paths from $u_i$ to $v$. 

Define $\hat{v}:=\max\{P_i\vee P\;|\;\ i\in[n]\}$, where the maximum is relative to $\le_P$. Notice that $\hat{v}\in V$ and is matched to a vertex $\hat{u}$ via an edge in $\mcM$ such that $\hat{v}<_P\hat{u}$. Then define 
$\widehat{G}_{\mcM}\langle P\rangle$ as the induced subgraph of $G_{\mcM}\langle P\rangle$ by removing those vertices in $y\in P_i$ such that $y\le_{P_i}\hat{v}$ for all $i\in [n]$. Note that $\hat{v}\notin \widehat{G}_{\mcM}\langle P\rangle$.
The definition of $\hat{v}$ is only well defined if $n>1$, otherwise we define $\widehat{G}_{\mcM}\langle P\rangle = G_{\mcM}\langle P\rangle$.

We similarly define $\check{u}:=\min\{P_i\wedge P\;|\;\;i\in[n]\}$, noting that $\check{u}\in U$ and is matched to a vertex $\check{v}\in V$ by $\mcM$ such that $\check{v}<_P\check{u}$. We define $\check{G}_{\mcM}\langle P\rangle$ as the subgraph of $G_{\mcM}\langle P\rangle$ induced by those vertices $\le_{P_i} \check{v}$ for $i\in[n]$. Note that $\check{u}\notin$ $\check{G}_{\mcM}\langle P\rangle$. 

\begin{observation}\label{obs:intersection}
For a bipartite graph $G=(U,V,E)$ and matching $\mcM$, then for any augmenting path $P$, $\hat{G}_{\mcM}\langle P\rangle \cap \check{G}_{\mcM}\langle P\rangle$ satisfies the conditions of Proposition \ref{prop:unique-ends}.
\end{observation}

\begin{lemma}\label{lem:one_u_suffices}
Let $G=(U,V,E)$ be a bipartite graph and $\mcM$ a matching. Let $P$ be an augmenting path, then 
$$|K_{\mcM}(G_{\mcM}\langle P\rangle)| = |K_{\mcM\triangle P}(G_{\mcM}\langle P\rangle)|\iff|K_{\mcM}(\widehat{G}_{\mcM}\langle P\rangle)| = |K_{\mcM\triangle P}(\widehat{G}_{\mcM}\langle P\rangle)|.$$ 
\end{lemma}
\begin{proof}

Let $\overline{G}_{\mcM}\langle P\rangle$ be the complement of $\widehat{G}_{\mcM}\langle P\rangle$ within $G_{\mcM}\langle P\rangle$. It suffices to show that $|K_{\mcM}(\overline{G}_{\mcM}\langle P\rangle)|=|K_{\mcM\triangle P}(\overline{G}_{\mcM}\langle P\rangle)|.$ There are two cases. The first is that there is only one unsaturated vertex in $U\cap G_{\mcM}\langle P\rangle$. Then $\overline{G}_{\mcM}\langle P\rangle=\emptyset$, trivially giving the desired result.

Now suppose there is more than one unsaturated vertex of $U$ in $\overline{G}_{\mcM}\langle P\rangle$. In that case, $Z_{\mcM}(\overline{G}_{\mcM}\langle P\rangle)=Z_{\mcM\triangle P}(\overline{G}_{\mcM}\langle P\rangle)=V\cap \overline{G}_{\mcM}\langle P\rangle$, from which we get the desired result. 
\end{proof}

Lemma \ref{lem:one_u_suffices} states we may assume without loss of generality that a graph of the form $G_{\mcM}\langle P\rangle$ for some augmenting path $P$ contains only a single unsaturated vertex $u\in U$.

\begin{proposition}\label{prop:counter_example}
Let $G=(U,V,E)$ be a bipartite graph and $\mcM$ a matching. Suppose there is an augmenting path $P$ such that $G_\mcM\langle P\rangle\setminus\check{G}_{\mcM}\langle P\rangle$ contains two distinct unsaturated vertices $v_1, v_2\in V$. Then $|K_{\mcM}(G)| > |K_{\mcM\triangle P}(G)|$.
\end{proposition}
\begin{proof}
Let $\tilde{G}_{\mcM}\langle P\rangle$ be the complement of $\check{G}_{\mcM}\langle P\rangle$ with respect to $G_{\mcM}\langle P\rangle$. Using Proposition \ref{prop:unique-ends}, Lemma \ref{lem:one_u_suffices}, and Observation \ref{obs:intersection}, is suffices to show that $|K_{\mcM}(\tilde{G}_{\mcM}\langle P\rangle)|>|K_{\mcM}(\tilde{G}_{\mcM}\langle P\rangle)|.$ 

Let $v=v_1,\dots,v_n$ be the unsaturated vertices of $\tilde{G}_{\mcM}\langle P\rangle$. Notice that for each $i\in [n]$, there is at least one path connecting $v_1$ and $v_i$ that passes through $\check{v}$. Denote these paths by $H^i_{j}$ for $1\le j\le k_i$. Let $H_i=\bigcup_{j=1}^{k_i}{H^{i}_j}.$
 Then note that 
$$P\cup\bigg(\bigcup_{i=2}^n{H_i}\bigg)= \tilde{G}_{\mcM}\langle P\rangle.$$

 $K_\mcM(G)\cap H^i_{j}$ includes both $v$ and $v_i$. Since a path from $v_1$ to $v_i$ must have even length, it has an odd number of vertices. We see that  $|K_{\mcM}(H^i_j)|=|H^i_j\setminus K_{\mcM}(G)| +1$.

Now we consider the matching $\mcM\triangle P$. Now neither $v, v_i$ are reachable by an alternating path in $G$ from some unsaturated vertex of $U$ with respect to $\mcM\triangle P$. Thus $K_{\mcM\triangle P}(H^i_j) = H^i_j\setminus K_{\mcM}(G)$, clearly reducing the number of vertices chosen on $H^i_j$ by K\H onig's Procedure by the equality shown in the previous paragraph. Then note that with respect to $\mcM\triangle P$, none of the vertices of $P$ are reachable by an augmenting path in $G$ either. Thus $K_{\mcM}(P)=P\setminus K_{\mcM\triangle P}(P)$ and since $P$ has an even number of vertices the cardinalities of the two sets agree. Thus we have $|K_{\mcM}(\tilde{G}_{\mcM}\langle P\rangle)| > |K_{\mcM\triangle P}(\tilde{G}_{\mcM}\langle P\rangle)|$ proving the result.

\end{proof}

We now present our main theorem  which allows us avoid extending maximal matchings in the case where several unsaturated vertices in $U$ all connect to a single unsaturated vertex in $V$ via augmenting paths. This reduces the number of times an algorithm must extend maximal matchings winning us computation time. Furthermore, the theorem states that this is the best we can do. In this sense, we give a complete classification of maximal matchings that produce minimum vertex covers via K\H onig's Procedure.
\begin{theorem}\label{thm:main-classification}
Let $\mcM$ be a maximal matching on a bipartite graph $G=(U,V,E)$. Then $K_{\mcM}(G)$ is a minimum vertex cover if and only if there is no augmenting path $P$ such that $G_{\mcM}\langle P\rangle\setminus\check{G}_\mcM\langle P\rangle$ contains more than a single unsaturated vertex of $V$.
\end{theorem}
\begin{proof}
First assume that there is no extendable augmenting path $P$ such that $G_{\mcM}\langle P\rangle$ contains more than a single unsaturated vertex of $V$. Then note that the vertex $x$ defined in the construction of $\widehat{G}_{\mcM}\langle P\rangle$ is in $V$ and is matched to a vertex we call $u\in U$. Thus we have that $\widehat{G}_{\mcM}\langle P\rangle$ satisfies the conditions of Proposition \ref{prop:unique-ends} giving us the first direction using Lemma \ref{lem:one_u_suffices}.
Secondly Proposition \ref{prop:counter_example} gives the other direction.
\end{proof}

\section{Conclusion}

We have seen in Theorem \ref{thm:reverse-konig} that every minimum vertex cover can be found by applying K\H{o}nig's procedure to a matching. We found it initially surprising that we may be forced to consider non-maximal matchings in order to generate all minimum vertex covers. 

As such, we introduced the notion of enumeratively K\H{o}nig-Egerv\'ary graphs where it is true that every minimum vertex cover can be recovered from a maximal matching using K\H{o}nig's procedure and gave a large family of examples. We would be very interested in a classification of enumeratively K\H{o}nig-Egerv\'ary graphs.

We also classified the obstacles in a matching that prevent K\H{o}nig's procedure from returning a minimum vertex cover in Theorem \ref{thm:sufficient-for-mvc}. We consider this a very interesting result, especially in the study of the counting and enumeration versions of these problems. For example, we would love a connection that allowed quick translations on upper bounds on maximal matchings to upper bounds on maximum independent sets. We hope that these results can be leveraged towards achieving such translations.

When the research began for this paper, the main motivation was algorithmic. We wished to know if minimum vertex cover algorithms could be sped up by weakening the kinds of matchings such an algorithm would need to find. While we did find a weaker condition, we found it too slow computationally to test for the condition in Theorem \ref{thm:sufficient-for-mvc}. 

We tried to modify the Hopcroft-Karp algorithm to not extend unsaturated vertices in $U$ belonging to a unique augmenting path. However, apart from the smallest graphs, this condition seems to never occur for random examples, incurring computational overhead with little prospect of reward.

We wonder if this can be made precise. For example, is it true that given a random bipartite graph and a random maximal matching that applying K\H{o}nig's procedure on it does not produce a minimum vertex cover? We strongly suspect this is true. This means that in essence, we should not expect to do better than finding maximum matchings when working with random bipartite graphs. While disappointing, it is our belief that it will be very unlikely that minimum vertex cover algorithms can be made more efficient by feeding a different structure in K\H{o}nig's procedure.

Nevertheless, connecting matchings and minimum vertex covers has often proven to be fruitful and this paper elaborates their relationship even further in the case of bipartite graphs, even if it raises as many questions as it answers.

\bibliography{bibfile}
\end{document}